\documentclass[12pt,reqno]{amsart}
\usepackage{hyperref}
\usepackage{amsmath,amsthm,amsopn,amssymb,a4wide,varioref}
\usepackage{color, bm, amscd, tikz-cd}
\newcommand{\cX}{{\mathcal X}}

\usepackage{shadow}

\newtheorem{theorem}{Theorem}
\newtheorem{thm}{Theorem}[section]

\newtheorem{lemma}[thm]{Lemma}

\newtheorem{proposition}[thm]{Proposition}

\theoremstyle{definition}

\newtheorem{remark}[thm]{Remark}
\newtheorem{note}[thm]{Note}

\numberwithin{equation}{section}

\def\textmatrix#1&#2\\#3&#4\\{\bigl({#1 \atop #3}\ {#2 \atop #4}\bigr)}
\def\dispmatrix#1&#2\\#3&#4\\{\left({#1 \atop #3}\ {#2 \atop #4}\right)}

\author{Daniel Alpay, Tirthankar Bhattacharyya, Abhay Jindal and Poornendu Kumar}
\address{Department of Mathematics\\
	Chapman University\\
	Orange, CA 92866, USA.}
\email{alpay@chapman.edu}

\address{Department of Mathematics\\
	Indian Institute of Science\\
	Bangalore 560012, India}
\email{tirtha@iisc.ac.in; abjayj@iisc.ac.in}

\address{Department of Mathematics\\
	University of Manitoba, Winnipeg, Canada, R3T 2N2}
\email{poornendukumar@gmail.com, Poornendu.Kumar@umanitoba.ca}

\usepackage{fancyhdr}
\begin{document}
\thanks{
	{\em Key words and phrase}: Schwarz lemma, Schur algorithm,  Poincar\'e contractivity, Multiplier algebras, Complete Nevanlinna Pick kernels, Indefinite inner products.}
	\title{Complete Nevanlinna-Pick kernels, the Schwarz lemma and the Schur algorithm}
	\maketitle
	\begin{abstract}
          We investigate the Schwarz lemma and the Schur algorithm for elements in the unit ball of the multiplier algebra of a reproducing kernel Hilbert space on the open unit ball whose kernel satisfies the complete Nevanlinna-Pick property. This paper also explores the Poincar\'e contractivity for elements in the unit ball of the multiplier algebra of a reproducing kernel Hilbert space whose kernel satisfies the property mentioned above. 
		\end{abstract}
	
                \section{Introduction}
                \setcounter{equation}{0}
                Let $s$ be a Schur function, that is a function analytic and bounded by $1$ in the open unit disc $\mathbb D$. Schwarz lemma asserts that if $s(0)=0$ then $s(z)=z\sigma(z)$ where $\sigma$ is itself a Schur function. Equivalently
                \[
|s(z)|\le |z|,\quad z\in\mathbb D.
                  \]
More generally for $a\in \mathbb D$ we have
\[
\frac{s(z)-s(a)}{1-s(z)\overline{s(a)}}=\frac{z-a}{1-z\overline{a}}\sigma(z)
  \]
  where as above $\sigma$ is a Schur function. Equivalently
  \[
s(z)=\frac{s(a)+b_a(z)\sigma(z)}{1+b_a(z)\sigma(z)\overline{s(a)}},\quad{\rm where}\quad b_a(z)=\frac{z-a}{1-z\overline{a}},
    \]
or, still equivalently, the {\em Poincar\'e contractivity} property of Schur functions
 \begin{equation}
  \label{schur-one-point}
    \left|\frac{s(z)-s(a)}{1-s(z)\overline{s(a)}}\right|\le \left|\frac{z-a}{1-z\overline{a}}\right|,\quad z\in\mathbb D.
    \end{equation}

Since the Banach algebra of bounded analytic functions on the open unit disc $\mathbb D$ is the multiplier algebra of the Hardy space, it is natural to ask if we can generalize the above to multiplier algebras of reproducing kernel Hilbert spaces on the open unit ball. A class of kernels which have been greatly studied in recent times is the complete Nevanlinna Pick (CNP) kernels. As the name suggests, these arose in relation to interpolation problems. The most prominent among all CNP kernels is the Drury-Arveson kernel
$$a_m(\lambda , \mu) = \frac{1}{1 -\langle \lambda , \mu \rangle}$$
defined in the open unit ball $\mathbb B_m$ of $\mathbb C^m$. Here $m$ is a cardinal and it is possible that it is $\aleph_0$ in which case, the base set $\mathbb B_m$ of the kernel is the open unit ball of $l^2$. A remarkable theorem relates any CNP kernel to an $a_m$ for some $m$. We quote parts of the theorem which we shall use in this paper. A  kernel $k$ on a set $\Omega$ is called {\em irreducible} if $k( z,  w)\neq 0$ for all $ z, w\in \Omega$, and $k_{ w}$ and $k_{  v}$ are linearly independent if $ v\neq  w$. Instead of defining a CNP kernel, we mention below a characterizing property which is the one we shall use. For details on CNP kernels, see \cite{AM}.

\begin{thm}[\cite{AM}, Theorem 8.2] \label{CNP-beta} Let $k$ be an irreducible kernel on a set $\Omega$. Then $k$ is CNP if and only if for some cardinal $m$, there is an injection $\beta : \Omega \rightarrow \mathbb B_m$ and a nowhere vanishing function $\delta$ on $\Omega$ such that
$$k(z,w) = \delta(z) \overline{\delta(w)} a_m(\beta(z) , \beta(w)).$$
If, in addition, there is a topology on $\Omega$ such that $k$ is continuous on $\Omega \times \Omega$, then $\beta$ is a continuous embedding of $\Omega$ into $\mathbb B_m$.
\end{thm}

A positive kernel $k$ on $\Omega$ gives rise to a reproducing kernel Hilbert space $\mathcal H_k$ in the sense of Chapter 2 of \cite{AM}. Let $\mathbb C^{p\times q}$ denote the algebra of $p\times q$ matrices. A $\mathbb C^{p\times q}$-valued function $s$ on $\Omega$ is called a {\em Schur multiplier} of the Hilbert space $\mathcal H_k$ if the kernel
                \begin{equation}
(I_p-s(z)s(w)^*) k(z , w) \label{smultiplier}
\end{equation}
is a positive kernel on $\Omega$. We shall use the notation $\mathcal M^{p\times q}_k$ for the set of such multipliers with $\mathcal M_k$ standing for $\mathcal M^{1\times 1}_k$. Setting $z=w$ we see that a Schur multiplier is always bounded by $1$. The converse is not always true. For instance, the function $\varphi(z) = z$ is not a Schur multiplier of the Dirichlet space. See also \cite[p. 6]{akap1}.

\begin{note} \label{notationHs}
Since the kernel in \eqref{smultiplier} is a positive kernel, there is a reproducing kernel Hilbert space associated with it. We shall denote it by $\mathcal H(s)$. \end{note}

We denote the reproducing kernel Hilbert space corresponding to $a_m$ by $\mathcal A_m$ and call it the Drury-Arveson space. An important example of an element of $\mathcal M^{1 \times m}_{a_m}$ is the {\em Blaschke factor}, defined for $\alpha\in\mathbb C^m$ by
\begin{equation} \label{bf}
b_\alpha(\lambda)=\frac{1-\langle \alpha, \alpha \rangle}{1-\langle \lambda,\alpha \rangle}(\lambda-\alpha )(I_m-\alpha^*\alpha)^{-1/2}.
  \end{equation}
  Note that for $\lambda,\mu\in\mathbb{B}_m$,
  \begin{equation}
    \label{yui}
    \frac{1-  b_a(\lambda) \; b_a(\mu)^* }{1-\langle \lambda , \mu\rangle}=\frac{1-\langle a,a\rangle}{(1-\langle \lambda,a\rangle)(1-\langle a,\mu\rangle}.
  \end{equation}
  See \cite[Theorem 2.2.2,  page 26]{rudin-ball} and \cite[Proposition 4.1, page 11]{akap1}.
  
The  multiplier algebra for the reproducing kernel Hilbert space arising out of a complete Nevanlinna-Pick kernel is yet to be well-understood. Unlike the classical case, where the multiplier algebra of the Hardy space is $H^\infty(\mathbb{D})$, Arveson showed in \cite{Arveson} that the multiplier algebra of the Drury-Arveson space is strictly smaller than $H^\infty(\mathbb{B}_d)$. A definitive characterization for the multiplier algebra of complete Nevanlinna-Pick kernels is currently unknown. Nevertheless, a realization formula exists (\cite{AM, btv}). Furthermore, even for polynomials $p$, the supremum norm of $p$ in general does not dominate the operator norm of $M_q$ on $H^2_d$, see \cite{Arveson, Shalit}. Explore further on this topic by referring to Fang's survey \cite{Fang} and \cite{AHMC, Clouatre-Hartz, Dru}. This paper is an effort to understand the multiplier algebra of a CNP space. 

There are numerous approaches to interpolation by Schur multipliers in the open unit disc and in the open unit ball. Here we will use the reproducing kernel approach
(see \cite{abds2,Dym_CBMS}). We will use the Schur algorithm as developed in \cite{abk} for the Drury-Arveson space. In the reproducing kernel approach to interpolation, three main steps can be distinguished (these same steps are of course present, in one way or the other, in the other approaches).

        \begin{enumerate}
        \item Build from the interpolation data a matrix/operator-valued function $\Theta$ with some metric properties (contractivity with respect to an indefinite metric).
        \item Show that a map is a contraction and deduce a linear fractional transformation which, {\it a  priori}, describes a part of the solutions.
        \item Show that one gets all solutions with the aforementioned linear fractional transformation.
        \end{enumerate}
The results start from Section \ref{SL-section}. Each section has a theorem, the main result of that section, with the same number as that of the section. In the first two sections, we shall discuss the Schwarz lemma for functions in the unit ball of the multiplier algebra of CNP kernels. Specifically, Section 1 will focus on a class of functions such that the function vanishes at some point, while Section 2 will address functions which need not vanish. To prove these results, we use the theory of indefinite inner product spaces. Following this, we shall proceed to establish Poincar\'e contractivity for elements in the unit ball of multiplier algebra which sheds some light on the multiplier algebra of CNP spaces. Finally, in Section 4, we shall present the Schur algorithm for functions in the unit ball of the multiplier algebra of CNP kernels.

\section{A Schwarz lemma at a point where the function vanishes} \label{SL-section}
\setcounter{equation}{0}

\begin{theorem}
  Let $\Omega$ be an open subset of the open unit ball, and let $k$ be a CNP kernel on $\Omega$ with its associated $\beta$ as in Theorem \ref{CNP-beta} above.
Then, $S \in \mathcal M_k$ if and only if $S(z)=G(\beta(z))$ for some $G \in \mathcal M_{a_m}$. Furthermore, $S(w_0)=0$ if and only if
\[
S(z)=b_{\beta(w_0)}(\beta(z))G_1(\beta(z))
\]
where $b_{\beta(w_0)}$ is a Blaschke factor as in \eqref{bf} above and $G_1 \in \mathcal M^{m \times 1}_{a_m}$.
\label{Maintheorem}
\end{theorem}
This section is devoted to the proof of the theorem above. Since a CNP kernel $k$ is fixed, it fixes $m$ and $\beta$. Hence, for notational brevity, we shall denote $\mathcal A_m$ by $\mathcal A$. We shall need to consider indefinite inner products on direct sums of the Drury-Arveson space. Let $\mathcal A^{p , q}$ denote the vector space direct sum of $p+q$  copies of $\mathcal A$ endowed with the indefinite inner product
\begin{equation*}
    [F,G]_{p\times q}= \sum_{i=1}^p \langle f_i , g_i \rangle_{\mathcal A} - \sum_{i=p+1}^q \langle f_i , g_i \rangle_{\mathcal A}
    \end{equation*}
    for $F=(f_1, \ldots , f_p , f_{p+1}, \ldots , f_{p+q})$ and $G=(g_1, \ldots , g_p , g_{p+1}, \ldots , g_{p+q})$. We present, with proofs, a special case of the computations from \cite{abk}. A more general case, needed to iterate the present procedure, is presented in the last section of this paper.

\begin{proposition}
Consider, for every $\nu$ in $\mathbb B_m$, the vector $ f_\nu$ in $\mathcal A^{1, 1}$ whose value at $\lambda$ in $\mathbb B_m$ is given by
  \begin{equation*}
  \label{fnu}
  f_\nu(\lambda)=\begin{pmatrix}\frac{1}{1-\langle \lambda,\nu\rangle}\\0\end{pmatrix}.
  \end{equation*}
  The one-dimensional space spanned by $f_\nu$ in $\mathcal A^{1, 1}$ is a one-dimensional reproducing kernel Hilbert space with reproducing kernel
\begin{equation*}
  K_{\Theta_\nu}(\lambda,\mu)=\frac{J-\Theta_\nu(\lambda)J\Theta_\nu(\mu)^*}{1-\langle \lambda,\mu\rangle} \; \text{ where }
\Theta_\nu(\lambda)=\begin{pmatrix}b_\nu(\lambda)&0\\
0_{1 \times m}&1\end{pmatrix}\; \text{ and } \;
J=\begin{pmatrix}I&0\\
0_{1 \times m}&-1\end{pmatrix}
  \end{equation*}
on $\mathbb C^m \oplus \mathbb C$ or $l^2 \oplus \mathbb C$.
\end{proposition}

\begin{proof} By the formula for the reproducing kernel of a finite dimensional reproducing Hilbert space we have to prove that
  \[
    \frac{J-\Theta_\nu(\lambda)J \Theta_\nu(\mu)^*}{1-\langle \lambda,\mu\rangle}=\
    \frac{f_\nu(\lambda)f_\nu(\mu)^*}{\|f_\nu\|^2_{\mathcal A^{1,1}}},\quad \lambda,\mu\in\mathbb B_N.
    \]
We have
\[
\begin{split}
\frac{J-\Theta_\nu(\lambda)J \Theta_\nu(\mu)^*}{1-\langle \lambda,\mu\rangle}
&=\frac{J-\begin{pmatrix}b_\nu(\lambda)&0\\ 0_{1\times m}&1\end{pmatrix}J \begin{pmatrix}b_\nu(\mu)&0\\0_{1\times m}&1\end{pmatrix}^*}{1-\langle \lambda,\mu\rangle}\\
&=\begin{pmatrix}\frac{1-b_\nu(\lambda)b_\nu(\mu)^*}{1-\langle \lambda,\nu\rangle}&0\\0&0\end{pmatrix}\\
&=\begin{pmatrix}\frac{1-\langle \nu,\nu\rangle}{(1-\langle \lambda,\nu\rangle)(1-\langle \mu,\nu\rangle)}&0\\0&0\end{pmatrix}\\
&=f_\nu(\lambda)f_\nu(\mu)^*
\end{split}
\]
which ends the proof since
\[
  \|f_\nu\|^2_{\mathcal A^{1,1}}=\left\|\frac{1}{1-\langle \lambda,\nu\rangle}\right\|^2_{\mathcal A}=\frac{1}{1-\langle \nu,\nu\rangle}.
\]
\end{proof}

Note that the one-dimensional space above is $J\mathcal H(\Theta_\nu)$ (recall Note \ref{notationHs}).

\begin{proposition}
  Let $s$ be an element of $\mathcal M_a$ and assume $s(\nu)=0$ at a point $\nu$ in $\mathbb B_m$. The map $\tau$ which sends
  \[
f\mapsto \begin{pmatrix}1, &-s\end{pmatrix}f
  \]
  is an isometry from $J\mathcal H(\Theta_\nu)$ into $\mathcal H(s)$, with adjoint
  \begin{equation}
    \label{adjoint}
    \tau^*\left(k_s(\cdot, \mu)\right)=K_{\Theta_\nu}(\cdot,\mu)\begin{pmatrix}1\\ \\ -\overline{s(\mu)}\end{pmatrix}
      \end{equation}
\end{proposition}

\begin{proof}
  \[
    (\tau(f_\nu))(\lambda)=\frac{1}{1-\langle \lambda,\nu\rangle}=    \frac{1-s(\lambda)\overline{s(\nu)}}{1-\langle \lambda,\nu\rangle}=k_s(\lambda,\nu)
\]
since $s(\nu)=0$. Thus
\[
  \|\tau(f_\nu)\|^2_{\mathcal H(k_s)}=k_s(\nu,\nu)=\frac{1}{1-\langle\nu,\nu\rangle},
\]
and hence the isometry property.  \smallskip

Formula \eqref{adjoint} is a special case of the formula for the adjoint of a multiplication operator between Hilbert spaces; see e.g. \cite[Corollary 5.22]{Ragh}.
\end{proof}

We shall now prove a result which resembles the Schwarz lemma.  For that, we shall need a theorem due to Leech. We state it below to the extent we need it and not in utmost generality.

        \begin{thm}[\cite{AM}, Theorem 8.57]
          Let $A$ and $B$ be respectively $\mathbb C^{p\times q}$-valued and $\mathbb C^{p\times r}$-valued functions continuous in $\mathbb B_m$ and assume that the kernel
          \begin{equation*}
           \label{AABB}
            \frac{A(\lambda)A(\mu)^*-B(\lambda)B(\mu)^*}{1-\langle \lambda,\nu\rangle}
          \end{equation*}
          is positive definite in $\mathbb B_m$. Then there exists a $\mathbb C^{q\times r}$-valued schur multiplier $C$ such that  $B=AC$.
          \label{leech-thm}
\end{thm}

\begin{lemma} Let $s$ be an element of $\mathcal M_a$ and let $\nu\in\mathbb B_m$. Then, $s(\nu)=0$ if and only if there exists a (in general not uniquely defined) $\mathcal M^{m \times 1}_{a_m}$ element $s_\nu$ such that
  \begin{equation}
    \label{inter-345}
s(\lambda)=b_\nu(\lambda)s_\nu(\lambda).
  \end{equation}
  \label{schur-ball}
\end{lemma}

\begin{proof} We divide the proof into steps. In the first three steps we assume that $s(\nu)=0$. The fourth step considers the converse statement.\\

  STEP 1: {\sl The kernel
    \[
          \frac{\begin{pmatrix}1&-s(\lambda)\end{pmatrix}\Theta_\nu(\lambda) J \Theta_\nu(\mu)^*\begin{pmatrix}1\\ \\ -\overline{s(\mu)}\end{pmatrix}}{1-\langle
            \lambda, \mu\rangle}
        \] is a positive kernel.
      }

      Since $I-\tau\tau^*\ge0$, the kernel $\langle (I-\tau\tau^*)k_s(\cdot,\mu),k_s(\cdot,\lambda\rangle_{\mathcal H(k_s)}$ is positive on $\mathbb B_m$. But, using \eqref{adjoint} we have
      \[
        \begin{split}
          \langle(\tau\tau^*)k_s(\cdot,\mu),k_s(\cdot,\lambda\rangle_{\mathcal H(k_s)}&=
\langle(\tau^*(k_s(\cdot,\mu)),\tau^*(k_s(\cdot,\lambda)\rangle_{\mathcal H(k_s)}\\
&=\left\langle K_{\Theta_\nu}(\cdot,\mu)\begin{pmatrix}1\\ \\ -\overline{s(\mu)}\end{pmatrix},K_{\Theta_\nu}(\cdot,\lambda)\begin{pmatrix}1\\ \\ -\overline{s(\lambda)}\end{pmatrix}\right\rangle_{\mathcal H(\Theta_\nu)}\\
&  =        \begin{pmatrix}1, &-s(\lambda)\end{pmatrix}K_{\Theta_\nu}(\lambda,\mu)\begin{pmatrix}1\\ \\ -\overline{s(\mu)}\end{pmatrix}\\
&=\frac{ \begin{pmatrix}1, &-s(\lambda)\end{pmatrix}(J-\Theta_\nu(\lambda)J \Theta_\nu(\mu)^*)\begin{pmatrix}1\\ \\ -\overline{s(\mu)}\end{pmatrix}}{1-\langle \lambda,\mu\rangle}\\
&=\frac{1-s(\lambda)\overline{s(\mu)}-\begin{pmatrix}1&-s(\lambda)\end{pmatrix}\Theta_\nu(\lambda)J \Theta_\nu(\mu)^*)\begin{pmatrix}1\\ \\ -\overline{s(\mu)}\end{pmatrix}}{1-\langle \lambda,\mu\rangle}\\
&=k_s(\lambda,\mu)-\frac{\begin{pmatrix}1&-s(\lambda)\end{pmatrix}\Theta_\nu(\lambda)J \Theta_\nu(\mu)^*)\begin{pmatrix}1\\ \\ -\overline{s(\mu)}\end{pmatrix}}{1-\langle \lambda,\mu\rangle}.
                    \end{split}
                  \]
                  Hence,
  \begin{equation}
    \begin{split}
      \langle (I-\tau\tau^*)k_s(\cdot,\mu),k_s(\cdot,\lambda\rangle_{\mathcal H(k_s)}
      &=
    \frac{\begin{pmatrix}1,&-s(\lambda)\end{pmatrix}\Theta_\nu(\lambda) J \Theta_\nu(\mu)^*\begin{pmatrix}1\\ \\-\overline{s(\mu)}\end{pmatrix}}{1-\langle
      \lambda, \mu\rangle}\\
    \end{split}
    \label{aa*-bb*}
    \end{equation}
    is a positive kernel.

    STEP 2: {\sl The kernel \eqref{aa*-bb*} can be rewritten as
      \begin{equation}
        \label{kernelbb*}
\frac{b_\nu(\lambda)b_\nu(\mu)^*-s(\lambda)\overline{s(\mu)}}{1-\langle
  \lambda, \mu\rangle}.
\end{equation}

}

We have
\[
\begin{pmatrix}1,&-s(\lambda)\end{pmatrix}\Theta_\nu(\lambda) = \begin{pmatrix}1&-s(\lambda)\end{pmatrix}
\begin{pmatrix}b_\nu(\lambda)&0\\
0_{1\times m}&1\end{pmatrix}=\left(
  \begin{matrix}\underbrace{b_\nu(\lambda)}_{1\times m}&-s(\lambda)\end{matrix}\right)
\]
and hence
\[
  \begin{split}
    \begin{pmatrix}1,&-s(\lambda)\end{pmatrix}\Theta_\nu(\lambda) J \Theta_\nu(\mu)^*\begin{pmatrix}1\\ \\-\overline{s(\mu)}\end{pmatrix}&=
\begin{pmatrix}b_\nu(\lambda)&-s(\lambda)\end{pmatrix}J\begin{pmatrix}\overline{b_\nu(\mu)}\\-\overline{s(\mu)}\end{pmatrix}
\\&=b_\nu(\lambda)b_\nu(\mu)^*-s(\lambda)\overline{s(\mu)}
  \end{split}
  \]
and hence the result.\\

STEP 3: {\sl There exists a $\mathbb C^{m\times 1}$-valued multiplier $s_\nu$ such that $b_\nu s_\nu=s$.}\\

It suffices to apply Theorem  \ref{leech-thm} to the positive kernel \eqref{kernelbb*}.\\

STEP 4: {\sl Any $s$ of the form \eqref{inter-345} is a Schur multiplier vanishing at $\nu$.}\\

It suffices to write
\[
  \begin{split}
    \frac{1-s(\lambda)\overline{s(\mu)}}{1-\langle\lambda,\mu\rangle}&=\frac{1-b_\nu(\lambda)b_\nu(\mu)^*}{1-\langle\lambda,\mu\rangle}+b_\nu(\lambda)\left(\frac{1-s_\nu(\lambda)s_\nu(\mu)^*}{1-\langle\lambda,\mu\rangle}\right)b_\nu(\mu)^*\\
  &=\frac{1-\langle\nu,\nu\rangle}{(1-\langle\lambda,\nu\rangle)(1-\langle \nu,\lambda\rangle)}+b_\nu(\lambda)\left(\frac{1-s_\nu(\lambda)s_\nu(\mu)^*}{1-\langle\lambda,\mu\rangle}\right)b_\nu(\mu)^*,
  \end{split}
\]
where we have used \eqref{yui}, and the above expresses $k_s$ as a sum of two positive kernels.
\end{proof}

{\em We are now ready to complete the proof of Theorem \ref{Maintheorem}.}  First note that any multiplier of $\mathcal H(k)$ is of the form $S(z)=s(\beta(z))$, where $s$ is a (in general not uniquely defined) multiplier of the Arveson space. This is just part 3 of Theorem 3.1 on page 105 of \cite{btv}. Next, Lemma \ref{schur-ball} above tells us that if $S(w_0)=0$ for some $w_0\in\Omega$, then there exists a (in general not uniquely defined) $s_1$ in $\mathcal M_a^{m \times 1}$  such that
  \begin{equation}
    \label{inter-un-point}
  s(\lambda)=b_{\beta(w_0)}(\lambda)s_1(\lambda).
  \end{equation}

Now, for the rest of the proof, we note that from \eqref{inter-un-point} we have
\[
S(z)=s(\beta(z))=b_{\beta(w_0)}(\beta(z))s_1(\beta(z)).
\]
To conclude we show that  $S_1(z)=s_1(\beta(z))$ is in $\mathcal M_{k_\beta}^{m,1}$. Let $k_\beta$ be defined as 
 \begin{equation*}
	K_\beta(z,w)=\frac{1}{1-\langle\beta(z),\beta(w)\rangle_{\mathbb C^N}},\quad z,w\in\Omega.
\end{equation*}. We have
\[
\frac{I_m-S_1(z)(S_1(w))^*}{1-\langle\beta(z),\beta(w)\rangle}=\frac{I_m-s_1(\beta(z))(s_1(\beta(w)))^*}{1-\langle\beta(z),\beta(w)\rangle}
  \]
which is positive definite on $\Omega$ since $s_1$ is in $\mathcal M_a^{m \times 1}$ and so the kernel
\[
\frac{I_m-s_1(\lambda)s_1(\mu)^*}{1-\langle\lambda,\mu\rangle},
\]
is positive definite in $\mathbb B_m$ and in particular on $\beta(\Omega)\subset\mathbb B_m$.

\begin{remark} If $\cX$ is a Banach space of analytic functions on a domain $\Omega\subset \mathbb{C}^d$, then one says that one can solve {\em Gleason problem} for  $\cX$ if, whenever $f\in\cX$ and $\lambda\in\Omega$, there exist functions $f_1, \dots, f_d \in \cX$ such that $ f(z)- f(\lambda) = \sum_{i=1}^{d} (z_i-\lambda_i) f_i(z)$. See \cite{rudin-ball} for more details on this. The Gleason problem is solvable in the Multiplier algebra of the Drury-Arveson space \cite{Gleason-Richter-Sundberg}. Later, Hartz in \cite{Hartz} solved the Gleason problem for the Multiplier algebra of complete Nevanlinna-Pick kernels. However, it should be noted that the solution is not established in the unit ball of the multiplier algebra if we start the function from the unit ball of the multiplier algebra.
	
It is connected to the Schwarz lemma in the context of scalar-valued functions. Notably, the Gleason problem, where the function vanishes at the point $0$ in the unit ball of the multiplier algebra of the Drury-Arveson space, is closely linked to  Schwarz Lemma. Specifically, Schwarz Lemma implies the solvability of the Gleason problem, and vice versa. At the point $0$ in $\mathbb{B}_d$, the Gleason problem is solvable in the unit ball of the multiplier algebra of the Drury-Arveson space, as indicated by Corollary 4.2 in \cite{Gleason-Richter-Sundberg}. While the theorem may not be explicitly formulated in this manner, the proofs therein substantiate this connection. Hence we have Schwarz Lemma at this point.
\end{remark}
\section{A Schwarz lemma at a point where the function need not vanish}
\setcounter{equation}{0}
\label{sec5}
We now present  the tangential version of the previous result.

 \begin{theorem} Let $s$ be a $\mathbb C^{p\times q}$-valued Schur multiplier and let $\nu\in\mathbb B_N$. Let $\xi\in\mathbb C^{p}$ and $\eta\in
            \mathbb C^{q}$ and assume $\eta^*\eta<\xi^*\xi$. Then $\xi^*s(\nu)=\eta^*$ if and only if there exists a (in general not uniquely defined)
            $\mathbb C^{(p-1+N)\times q}$-valued Schur multiplier $s_\nu$ such that
  \begin{equation}
    \label{inter-3456}
    \begin{split}
    s(\lambda)&=\left(\begin{pmatrix}\frac{\xi b_\nu(\lambda)}{\sqrt{c^*J_{p,q}c}}&\begin{pmatrix}U_1\\U_3\end{pmatrix}\end{pmatrix}
    s_\nu(\lambda)+\frac{\xi\eta^*}{c^*J_{p,q}c}\left(I_q+\frac{\eta\eta^*}{c^*J_{p,q}c}\right)^{-1/2}\right)\times\\
    &\hspace{5mm}\times\left(\begin{pmatrix}\frac{\eta b_\nu(\lambda)}{\sqrt{c^*J_{p,q}c}}& 0\end{pmatrix}
      s_\nu(\lambda)+\left(I_q+\frac{\eta\eta^*}{c^*J_{p,q}c}\right)^{1/2}\right)^{-1},
      \end{split}
\end{equation}
\label{schur-arveson}
\end{theorem}

\begin{proof} We start the proof with a $p\times p$ unitary matrix $U = \textmatrix U_1&U_2\\U_3&U_4\\$ where $U_1$ is a ${(p-1)\times(p-1)}$ matrix  and 
\[
\begin{pmatrix}U_2\\U_4\end{pmatrix}=\frac{\xi}{\sqrt{\xi^*\xi}}, \text{ or in other words, }
U\begin{pmatrix}0\\0\\ \vdots\\ 1\end{pmatrix}=\frac{\xi}{\sqrt{\xi^*\xi}}.
 \]
Because $U$ is unitary, the last column is orthogonal to the first $p-1$ columns, i.e.
  \begin{equation}
    \xi^*\begin{pmatrix}U_1\\U_3\end{pmatrix}=0.
    \label{rtyuiop}
    \end{equation}

We set
\[
  J_{p,q}=\begin{pmatrix}I_p&0\\0&-I_q\end{pmatrix},\quad c=\begin{pmatrix}\xi \\ \eta\end{pmatrix},\quad {\rm and}\quad
f_\nu(\lambda)=\frac{\begin{pmatrix}\xi\\ \eta\end{pmatrix}}{1-\langle \lambda,\nu\rangle}\in\mathcal A^{p+q}_{J_{p,q}}.
\]
It holds that
\[
  \|f_\nu\|^2_{\mathcal A^{p+q}_{J_{p,q}}}=\frac{c^*J_{p,q}c}{1-\langle \nu,\nu\rangle}=\frac{\xi^*\xi-\eta^*\eta}{1-\langle \nu,\nu\rangle}.
\]
We can write
\begin{equation}
  \label{jjstar}
  \begin{split}
    \frac{f_\nu(\lambda)f_\nu(\mu)^*}{  \|f_\nu\|^2_{\mathcal A^{p+q}_{J_{p,q}}}}&=\frac{1}{c^*Jc}c\left(\frac{1-\langle\nu,\nu\rangle}{(1-\langle \lambda,\nu\rangle)(1-\langle \mu,\nu\rangle)}\right)c^*\\
    &=\frac{1}{c^*Jc}c\left(\frac{1-b_\nu(\lambda)b_\nu(\mu)^*}{1-\langle\lambda,\mu\rangle}\right)c^*\\
    &=\frac{J_{p,q}-\left(J_{p,q}-\frac{cc^*}{c^*Jc}+\frac{cb_\nu(\lambda)b_\nu(\mu)^*c^*}{c^*Jc}\right)}{1-\langle\lambda,\mu\rangle}.
  \end{split}
\end{equation}
To continue, we need to find the signature and a diagonalization of $J_{p,q}-\frac{cc^*}{c^*Jc}$. In the following lemmas which are parts of the proof of Theorem \ref{schur-arveson}, $\xi \in \mathbb C^p$ and $\eta \in \mathbb C^q$ satisfy $\eta^*\eta < \xi^*\xi$. Let $U$ be as defined theorem above.
\begin{lemma}
Let $\alpha$ be the $(p+q)\times (p+N-1+q)$ matrix
\begin{equation*}
  \alpha=\begin{pmatrix}\begin{pmatrix}U_1\\U_3\end{pmatrix}&\frac{\xi\eta^*}{c^*Jc}(I_q+\frac{\eta\eta^*}{c^*Jc})^{-1/2}\\
0&
    (I_q+\frac{\eta\eta^*}{c^*Jc})^{1/2}\end{pmatrix}.
  \end{equation*}
  Then   \begin{equation}
J_{p,q}-\frac{cc^*}{c^*Jc}=\alpha J_{p+N-1,q}\alpha^*
  \label{JC*C}
  \end{equation}
\end{lemma}

\begin{proof}
Let
$$ A = I_p-\frac{\xi\xi^*}{c^*Jc}, \; B = -\frac{\xi\eta^*}{c^*Jc},  \text{ and } D = -I_q-\frac{\eta\eta^*}{c^*Jc}.$$
Then
\[
  \begin{split}
    A-BD^{-1}B^*&=I_p-\frac{\xi\xi^*}{c^*Jc}+\frac{\xi\eta^*\left(I_q+\frac{\eta\eta^*}{c^*Jc}\right)^{-1}\eta\xi^*}{(c^*Jc)^2}\\
    &=I_p-\frac{1}{c^*Jc}\xi\underbrace{\left(1-\eta^*\left(I_q+\frac{\eta\eta^*}{c^*Jc}\right)^{-1}\eta \frac{1}{c^*Jc}\right)}_{=\left(\frac{\xi^*\xi}{c^*Jc}\right)^{-1}}\xi^*=I_p-\frac{\xi\xi^*}{\xi^*\xi}.
    \end{split}
  \]
Thus,
\[
  \begin{split}
    J_{pq}-\frac{cc^*}{c^*Jc}
    &=\begin{pmatrix}I_p-\frac{\xi\xi^*}{c^*Jc}&-\frac{\xi\eta^*}{c^*Jc}\\
    -\frac{\eta\xi^*}{c^*Jc}&-I_q-\frac{\eta\eta^*}{c^*Jc}\end{pmatrix}\\
  &{=}\begin{pmatrix} A&B\\B^*&D\end{pmatrix}\\
  &=\begin{pmatrix}I_p&BD^{-1}\\0&I_q\end{pmatrix}
\begin{pmatrix}A-BD^{-1}B^*&0\\0&D\end{pmatrix}
  \begin{pmatrix}I_p&BD^{-1}\\0&I_q\end{pmatrix}^*\\
      &=\begin{pmatrix}I_p&\frac{\xi\eta^*}{c^*Jc}(I_q+\frac{\eta\eta^*}{c^*Jc})^{-1}\\0&I_q\end{pmatrix}\times\\
      &\hspace{5mm}\times
        \begin{pmatrix}I_p-\frac{\xi\xi^*}{\xi^*\xi}&0\\0&-\left(I_q+\frac{\eta\eta^*}{c^*Jc}\right)\end{pmatrix}
        \begin{pmatrix}I_p&\frac{\xi\eta^*}{c^*Jc}(I_q+\frac{\eta\eta^*}{c^*Jc})^{-1}\\0&I_q\end{pmatrix}^*.
   \end{split}
    \]
    To conclude, note that 
    \[
    \begin{split}
      I_p-\frac{\xi\xi^*}{\xi^*\xi}&=U\left(I_p-\begin{pmatrix}0\\0\\ \vdots\\ 1\end{pmatrix}
      \begin{pmatrix}0\\0\\ \vdots\\ 1\end{pmatrix}^*\right)U^* =U\begin{pmatrix}I_{p-1}&0_{(p-1)\times 1}\\       0_{1\times (p-1)}&0\end{pmatrix}U^*=\underbrace{\begin{pmatrix}U_1\\U_3\end{pmatrix}}_{\in\mathbb C^{p\times (p-1)}}\begin{pmatrix}U_1\\U_3\end{pmatrix}^*.
        \end{split}
    \]
\end{proof}

\begin{lemma} The equality
    \begin{equation*}
        \frac{f_\nu(\lambda)f_\nu(\mu)^*}{  \|f_\nu\|^2_{\mathcal A^{p+q}_{J_{p,q}}}}=\frac{J-\Theta_\nu(\lambda)\widetilde{J}\Theta_\nu(\mu)^*}{1-\langle
          \lambda,\mu\rangle}
      \end{equation*}
      holds with
      \begin{equation}\label{qazxc}
        \Theta_\nu(\lambda)=          \begin{pmatrix}\frac{1}{\sqrt{c^*Jc}}cb_\nu(\lambda)&\alpha            \end{pmatrix}
      \end{equation}
  and
\[
      \widetilde{J}=\begin{pmatrix}I_N&0\\ 0&J_{p-1,q}\end{pmatrix}=\begin{pmatrix}I_{N+p-1}&0\\ 0&-I_q\end{pmatrix}=J_{p-1+N,q}.
      \]
\end{lemma}

\begin{proof}
From \eqref{jjstar} we have:
\begin{equation*}
  \begin{split}
    \frac{f_\nu(\lambda)f_\nu(\mu)^*}{  \|f_\nu\|^2_{\mathcal A^{p+q}_{J_{p,q}}}}&=
    \frac{J-\left(J-\frac{cc^*}{c^*Jc}+\frac{cb_\nu(\lambda)b_\nu(\mu)^*c^*}{c^*Jc}\right)}{1-\langle\lambda,\mu\rangle}\\
    &=   \frac{J_{p,q}-\alpha I_{p-1}\alpha^*-\frac{cb_\nu(\lambda)b_\nu(\mu)^*c^*}{c^*Jc}}{1-\langle\lambda,\mu\rangle}\\
         &=   \frac{J_{p,q}-\Theta_\nu (\lambda)J_{p-1+N,q}\Theta_\nu(\mu)^*}{1-\langle\lambda,\mu\rangle},
  \end{split}
\end{equation*}
with $\Theta_\nu$ given by \eqref{qazxc}.

  \end{proof}
      \begin{lemma} It holds that
        \begin{equation*}
          \label{poiuy}
          c^*J\Theta_\nu(\nu)=0
        \end{equation*}
        \end{lemma}
      \begin{proof} Recall that $\alpha$ has been defined in \eqref{JC*C}. Since $b_\nu(\nu)=0_{1\times N}$ we need to show that $ c^*J_{p,q}\alpha=0_{1\times (p+q)}$.
        We have
          \[
            \begin{split}
              c^*J_{p,q}\alpha &=\begin{pmatrix}\xi^*&-\eta^*\end{pmatrix}
              \begin{pmatrix}\begin{pmatrix}U_1\\U_3\end{pmatrix}&\frac{\xi\eta^*}{c^*J_{p,q}c}(I_q+\frac{\eta\eta^*}{c^*J_{p,q}c})^{-1/2}\\
0&
(I_q+\frac{\eta\eta^*}{c^*J_{p,q}c})^{1/2}\end{pmatrix}\\
                          &=\begin{pmatrix}\xi^*\begin{pmatrix}U_1\\U_3\end{pmatrix}&\left(\frac{\xi^*\xi\eta^*}{c^*J_{p,q}c}-\eta^*(I_q+\frac{\eta\eta^*}{c^*J_{p,q}c})\right)
              (I_q+\frac{\eta\eta^*}{c^*J_{p,q}c})^{-1/2}
              \end{pmatrix}.
                            \end{split}
            \]
            The first block is equal to $0$ by \eqref{rtyuiop} and the second by direct computation:
            \[
              \frac{\xi^*\xi\eta^*}{c^*J_{p,q}c}-\eta^*\left(I_q+\frac{\eta\eta^*}{c^*J_{p,q}c}\right)=\frac{\xi^*\xi-c^*J_{p,q}c-\eta^*\eta}{c^*J_{p,q}c}\eta^*=0.
\]
      \end{proof}

The proof of Theorem \ref{schur-arveson} is nearly complete. If $s$ answers the question, the map
  \[
f\mapsto \begin{pmatrix}I_p,&-s\end{pmatrix}f
  \]
  sends $\mathcal H(\Theta_\nu)$ isometrically into $\mathcal H(S)$ and the linear fractional transformation follows from Leech theorem as in the proof of
  Lemma \ref{schur-ball}. Conversely, let $s_\nu$ be as in the theorem. Then, by \eqref{inter-3456} and using \eqref{rtyuiop}

  \[
  \begin{split}
    \xi^*s(\nu)&=\xi^*\frac{\xi\eta^*}{c^*J_{p,q}c}
    \left(I_q+\frac{\eta\eta^*}{c^*J_{p,q}c}\right)^{-1/2}\left(I_q+\frac{\eta\eta^*}{c^*J_{p,q}c}\right)^{-1/2}\\
    &=\frac{1}{c^*J_{p,q}c}
    \xi^*\xi\eta^*\left(I_q+\frac{\eta\eta^*}{c^*J_{p,q}c}\right)^{-1}\\
    &   =\xi^*\xi\left(1+\frac{\eta^*\eta}{c^*J_{p,q}c}\right)^{-1}\eta^*\\
    &=\frac{1}{c^*J_{p,q}c}\xi^*\xi\left(\frac{\xi^*\xi}{c^*J_{p,q}c}\right)^{-1}\eta^* =\eta^*
    \end{split}
\]

\end{proof}

\begin{remark}
  When $N=1$ and $p=q=1$ (scalar case for the open unit disc), with $\xi=1$ and $\eta=\overline{s(a)}$, set $\lambda=z\in\mathbb D$ and $s_\nu=\sigma$. We then have
  \[
b_a(z)=\frac{z-a}{1-z\overline{a}},\quad c^*J_{p,q}c=1-|s(a)|^2,\quad \left(I_q+\frac{\eta\eta^*}{c^*J_{p,q}c}\right)^{-1/2}=\sqrt{1-|s(a)|^2}.
\]
Furthermore, $U_1$ and $U_2$ do not appear in \eqref{inter-3456}, and the latter
formula reduces to the well-known formula \eqref{schur-one-point}
\begin{equation}
  \label{schur-one-point-1}
s(z)=\frac{b_a(z)\sigma(z)+s(a)}{1+b_a(z)\sigma(z)}
\end{equation}
where $\sigma$ varies among the family of all Schur functions.
\end{remark}

\section{Poincar\'e contractivity for $m>1$}
\setcounter{equation}{0}
The counterpart of Poincar\'e contractivity is now:
\begin{theorem}
  \label{11-August-2023-3}
  Let $s$ be a $\mathbb C^{p\times q}$-valued Schur multiplier and let $\nu\in\mathbb B_N$. Let $\xi\in\mathbb C^{p}$ and $\eta\in
            \mathbb C^{q}$ be such that $\xi^*s(\nu)=\eta^*$. Then
\begin{equation}
  \label{point}
    \left\|(\xi^*s(\lambda)-\eta^*)\left(I_q+\frac{\eta\eta^*}{c^*J_{p,q}c}\right)^{1/2}\right\|\le
   \frac{|\xi^*\xi-\xi^*s(\lambda)\eta|}{\sqrt{c^*J_{p,q}c}}\cdot\|b_\nu(\lambda)\|.
  \end{equation}
\end{theorem}
\begin{proof}
By virtue of $\xi^*s(\nu)=\eta^*$ being true, we shall use Theorem \ref{schur-arveson} via \eqref{inter-3456}.
Let
\[
  \Theta_\nu=\begin{pmatrix}A&B\\C&D\end{pmatrix}
\]
with
\[
  \begin{split}
    A(\lambda)&=\begin{pmatrix}\frac{\xi b_\nu(\lambda)}{\sqrt{c^*J_{p,q}c}}&\begin{pmatrix}U_1\\U_3\end{pmatrix}\end{pmatrix}\\
B(\lambda)&=\frac{\xi\eta^*}{c^*J_{p,q}c}\left(I_q+\frac{\eta\eta^*}{c^*J_{p,q}c}\right)^{-1/2}\\
C(\lambda)&=\begin{pmatrix}\frac{\eta b_\nu(\lambda)}{\sqrt{c^*J_{p,q}c}}& 0\end{pmatrix}\\
    D(\lambda)&=\left(I_q+\frac{\eta\eta^*}{c^*J_{p,q}c}\right)^{1/2},
\end{split}
\]
so that $s(\lambda)=(A(\lambda)s_\nu(\lambda)+C(\lambda))(C(\lambda)s_\nu(\lambda)+D(\lambda))^{-1}$.
From formula \eqref{inter-3456} we can write
\[
  \begin{split}
    \xi^*A(\lambda)&=\begin{pmatrix}\frac{\xi^*\xi b_\nu(\lambda)}{\sqrt{c^*J_{p,q}c}}&\begin{pmatrix}0\\0\end{pmatrix}\end{pmatrix}\quad\hspace{1cm}\quad({\rm since}\,\, \xi^*\begin{pmatrix}U_1\\U_3\end{pmatrix}=0)\\
    \xi^*B(\lambda)&=\frac{\xi^*\xi\eta^*}{c^*J_{p,q}c}\left(I_q+\frac{\eta\eta^*}{c^*J_{p,q}c}\right)^{-1/2}\\
    \eta^*C(\lambda)&=\begin{pmatrix}\frac{\eta^*\eta b_\nu(\lambda)}{\sqrt{c^*J_{p,q}c}}& 0\end{pmatrix}\\
    \eta^*D(\lambda)&=\eta^*\left(I_q+\frac{\eta\eta^*}{c^*J_{p,q}c}\right)^{1/2}.
    \end{split}
  \]
  Hence
  \[
    \begin{split}
      \xi^*A(\lambda)-\eta^*C(\lambda)&=\begin{pmatrix}\frac{\xi^*\xi b_\nu(\lambda)}{\sqrt{c^*J_{p,q}c}}&\begin{pmatrix}0\\0\end{pmatrix}\end{pmatrix}-\begin{pmatrix}\frac{\eta^*\eta b_\nu(\lambda)}{\sqrt{c^*J_{p,q}c}}& 0\end{pmatrix}\\
&=\begin{pmatrix}\frac{(\xi^*\xi-\eta^*\eta) b_\nu(\lambda)}{\sqrt{c^*J_{p,q}c}}&\begin{pmatrix}0\\0\end{pmatrix}\end{pmatrix}\\
 &=\begin{pmatrix}\sqrt{c^*J_{p,q}c}b_\nu(\lambda)&0\end{pmatrix}
    \end{split}
  \]
      and
      \[
    \begin{split}
      \xi^*B(\lambda)-\eta^*D(\lambda)&=\left(\frac{\xi^*\xi\eta^*}{c^*J_{p,q}c}-\eta^*\left(I_q+\frac{\eta\eta^*}{c^*J_{p,q}c}\right)\right)\left(I_q+\frac{\eta\eta^*}{c^*J_{p,q}c}\right)^{-1/2}\\
      &=\left(\frac{\xi^*\xi-c^*J_{p,q}c+\eta^*\eta}{c^*J_{p,q}c}\right)\eta^*\left(I_q+\frac{\eta\eta^*}{c^*J_{p,q}c}\right)^{-1/2}\\
      &=0.
    \end{split}
  \]

  Thus
  \[
    \begin{split}
      \left(\xi^*s(\lambda)-\eta^*\right)&=(\xi^*A(\lambda)s_\nu(\lambda)+\xi^*B(\lambda))(C(\lambda)s_\nu(\lambda)+D(\lambda))^{-1}-\eta^*\\
      &=\left((\xi^*A(\lambda)s_\nu(\lambda)+\xi^*B(\lambda))-(\eta^*C(\lambda)s_\nu(\lambda)+\eta^*D(\lambda))\right)\times\\
      &\hspace{5mm}\times(C(\lambda)s_\nu(\lambda)+D(\lambda))^{-1}\\
      &=\left((\xi^*A(\lambda)-\eta^*C(\lambda))s_\nu(\lambda)+(\xi^*B(\lambda))-\eta^*D(\lambda))\right)\times\\
      &\hspace{5mm}\times(C(\lambda)s_\nu(\lambda)+D(\lambda))^{-1}\\
      &=\left(\begin{pmatrix}\sqrt{c^*J_{p,q}c}b_\nu(\lambda)&0\end{pmatrix})s_\nu(\lambda)\right)\times\\
      &\hspace{5mm}\times
      \left(\begin{pmatrix}\frac{\eta b_\nu(\lambda)}{\sqrt{c^*J_{p,q}c}}& 0_{p\times(p-1)}\end{pmatrix}
      s_\nu(\lambda)+\left(I_q+\frac{\eta\eta^*}{c^*J_{p,q}c}\right)^{1/2}\right)^{-1}.
    \end{split}
  \]

  So

  \begin{equation*}
    \begin{split}
      \left(  (\eta^*-\xi^*s(\lambda))\begin{pmatrix}\frac{\eta b_\nu(\lambda)}{\sqrt{c^*J_{p,q}c}}& 0_{p\times(p-1)}\end{pmatrix}+\begin{pmatrix}\sqrt{c^*J_{p,q}c}b_\nu(\lambda)&0_{p\times (p-1)}
        \end{pmatrix}\right)s_\nu(\lambda)&=\\
      &\hspace{-4cm}=(\xi^*s(\lambda)-\eta^*)\left(I_q+\frac{\eta\eta^*}{c^*J_{p,q}c}\right)^{1/2}.
    \end{split}
  \end{equation*}
But
  \[
    \begin{split}
      \left(  (\eta^*-\xi^*s(\lambda))\begin{pmatrix}\frac{\eta b_\nu(\lambda)}{\sqrt{c^*J_{p,q}c}}& 0_{p\times (p-1)}\end{pmatrix}+\begin{pmatrix}\sqrt{c^*J_{p,q}c}b_\nu(\lambda)&0_{1\times (p-1)}\end{pmatrix}\right)&=\\
&\hspace{-4cm} =     \begin{pmatrix}\left(\frac{\eta^*\eta-\xi^*s(\lambda)\eta}{\sqrt{c^*J_{p,q}c}}+\sqrt{c^*J_{p,q}c}\right)b_\nu(\lambda)&0_{1\times (p-1)}\end{pmatrix}\\
      &\hspace{-4cm}=\begin{pmatrix}\left(\frac{\eta^*\eta-\xi^*s(\lambda)\eta+c^*J_{p,q}c}{\sqrt{c^*J_{p,q}c}}\right)b_\nu(\lambda)&0_{1\times (p-1)}\end{pmatrix}\\
      &\hspace{-4cm}=        \frac{\xi^*\xi-\xi^*s(\lambda)\eta}{\sqrt{c^*J_{p,q}c}}\begin{pmatrix}b_\nu(\lambda)&0_{1\times (p-1)}\end{pmatrix}
\end{split}
    \]

from which follows \eqref{point} since $s_\nu$ is contractive.
\end{proof}

\begin{remark}
When $p=q=N=1$ and $\eta=\overline{s(a)}$ and $\xi=1$,  this reduces to (with $z=\lambda$)
\[
|(s(z)-s(a)|\frac{1}{\sqrt{1-|s(a)|^2}}\le\left|\frac{1-s(z))\overline{s(a)}}{\sqrt{1-|s(a)|^2}}b_a(z)\right|,
  \]
i.e.
\[
  |s(z)-s(a)|\le|b_a(z)|\cdot \left|1-s(z))\overline{s(a)}\right|,
\]
which is \eqref{schur-one-point}.
\end{remark}

\begin{remark} We can generalize Theorem \ref{11-August-2023-3} to several points instead of one.
Let
  \[      \alpha(\lambda)=\frac{\xi^*\xi-\xi^*s(\lambda)\eta}{\sqrt{c^*J_{p,q}c}}\begin{pmatrix}b_\nu(\lambda),&0_{1\times (p-1)}\end{pmatrix} \text{ and }
      \beta(\lambda)=(\xi^*s(\lambda)-\eta^*)\left(I_q+\frac{\eta\eta^*}{c^*J_{p,q}c}\right)^{1/2}.
      \]
Then, $\alpha$ is $\mathbb C^{1\times N+p-1}$-valued and $\beta$ is $\mathbb C^{1\times q}$-valued.   Since $\alpha(\lambda)s_\nu(\lambda)=\beta(\lambda)$,
  the kernel
  \[
\frac{\alpha(\lambda)\alpha(\mu)^*-\beta(\lambda)\beta(\mu)^*}{1-\langle \lambda,\mu\rangle}=\alpha(\lambda)\frac{I_{N+p-1}-s_\nu(\lambda)s_\nu(\mu)^*}{1-\langle \lambda,\mu\rangle}\alpha(\mu)^*
    \]
    is positive definite.
    So, for any $M\in\mathbb N$ and for any $\lambda_1,\ldots, \lambda_M\in\mathbb B_N$ and $u_1,\ldots, u_M\in \mathbb C$, we get the inequality
    \[
      \sum_{m,n=1}^M u_n^*\frac{\alpha(\lambda_n)\alpha(\lambda_m)^*-\beta(\lambda_n)\beta(\lambda_m)^*}{1-\langle \lambda_n,\lambda_m\rangle}u_m\ge 0
\]
      which leads the matrix versions of the Poincar\'e contractivity

      \[
        \left(
          \frac{\alpha(\lambda_n)\alpha(\lambda_m)^*}{1-\langle \lambda_n,\lambda_m\rangle}\right)_{n,m=1}^M\ge
        \left(\frac{\beta(\lambda_n)\beta(\lambda_m)^*}{1-\langle \lambda_n,\lambda_m\rangle}\right)_{n,m=1}^M
      \]

    i.e.

    \begin{equation*}
      \begin{split}
      \left(
        \frac{(\xi^*\xi-\xi^*s(\lambda_n)\eta)(\xi^*\xi-\xi^*s(\lambda_m)\eta)^*}{c^*J_{p,q}c}
        \frac{b_\nu(\lambda_n)b_\nu(\lambda_m)^*}{1-\langle \lambda_n,\lambda_m\rangle}\right)_{n,m=1}^M&\ge\\
        &\hspace{-6cm}\ge\left(
          \frac{(\xi^*s(\lambda_n)-\eta^*)\left(I_q+\frac{\eta\eta^*}{c^*J_{p,q}c}\right)(s(\lambda_m)\xi-\eta)}{1-\langle \lambda_n,\lambda_m\rangle}
        \right)_{n,m=1}^M
              \end{split}
\end{equation*}

              or

              \begin{equation*}
                \begin{split}
      \hspace{-2cm}\left(  \frac{b_\nu(\lambda_n)b_\nu(\lambda_m)^*}{1-\langle \lambda_n,\lambda_m\rangle}\right)_{n,m=1}^M&\ge\\
 &\hspace{-3cm}\frac{1}{c^*J_{p,q}c}\left(
        \frac{  \frac{\xi^*s(\lambda_n)-\eta^*)}{(\xi^*\xi-\xi^*s(\lambda_n)\eta)}\left(I_q+\frac{\eta\eta^*}{c^*J_{p,q}c}\right)
            \frac{(s(\lambda_m)\xi-\eta)}{(\xi^*\xi-\xi^*s(\lambda_m)\eta)^*}}{1-\langle \lambda_n,\lambda_m\rangle}\right)_{n,m=1}^M.
    \end{split}
\end{equation*}
\end{remark}

\section{Schur algorithm and Poincar\'e contractivity for multipliers of $\mathcal H(K_\beta)$}
    \setcounter{equation}{0}
       \label{sec6}
        Let $\Omega$ be an open subset of the open unit ball, and let $\beta$ be  a $\mathbb B_N$-valued function continuous on $\Omega$. Let
  $K_\beta(z,w)$ denote the kernel
  \begin{equation}
    \label{kbeta}
  K_\beta(z,w)=\frac{1}{1-\langle\beta(z),\beta(w)\rangle_{\mathbb C^N}},\quad z,w\in\Omega.
\end{equation}
Theorem \ref{schur-arveson} allows to write a countertpart of Theorem \ref{Maintheorem} for multipliers of the reproducing kernel Hilbert space $\mathcal H(K_\beta)$ with reproducing kernel \eqref{kbeta}.
       \begin{theorem}
         Let $S$ be a $\mathbb C^{p\times q}$-valued Schur multiplier of $\mathcal H(K_\beta)$, with representation $S(z)G(\beta(z))$, $G$ being a Schur multiplier of $\mathcal A^q$. Let $w_0\in\Omega$ and assume that $\xi^*S(w_0)S(w_0)^*\xi<\xi^*\xi$. Then, in the notation of the previous theorem,
           \begin{equation}
    \label{inter-3456}
    \begin{split}
    S(z)&=\left(\begin{pmatrix}\frac{\xi b_{\beta(w_0)}(\beta(z))}{\sqrt{c^*J_{p,q}c}}&\begin{pmatrix}U_1\\U_3\end{pmatrix}\end{pmatrix}
    s_{\beta(w_0)}(\beta(z))+\frac{\xi\eta^*}{c^*J_{p,q}c}\left(I_q+\frac{\eta\eta^*}{c^*J_{p,q}c}\right)^{-1/2}\right)\times\\
    &\hspace{5mm}\times\left(\begin{pmatrix}\frac{\eta b_{\beta(w_0)}(\beta(z))}{\sqrt{c^*J_{p,q}c}}
& 0\end{pmatrix}
      s_{\beta(w_0)}(\beta(z))+\left(I_q+\frac{\eta\eta^*}{c^*J_{p,q}c}\right)^{1/2}\right)^{-1}
      \end{split}
\end{equation}
\end{theorem}

The proof follows the same lines as in the proof of Theorem \ref{Maintheorem} and will be omitted.\smallskip

One can iterate the preceding procedure; the Schur parameters are each time of a larger size, but if $N=1$. In the latter case, one gets back to the Schur algorithm developped in \cite{MR1246809}. One obtains Blaschke products of the kind studied in \cite{MR3735586}.

The counterpart of Poincar\'e contractivity is now:
\begin{equation*}
    \left\|(\xi^*S(z)-\eta^*)\left(I_q+\frac{\eta\eta^*}{c^*J_{p,q}c}\right)^{1/2}\right\|\le    \frac{|\xi^*\xi-\xi^*S(z)\eta|}{\sqrt{c^*J_{p,q}c}}\cdot\|b_{\beta(w_0)}(z)\|.
  \end{equation*}

  Here too,  the proof follows the same lines as in the proof of Theorem \ref{11-August-2023-3} and is omitted.\smallskip

\textbf{Acknowledgement:}
D. Alpay thanks the Foster G. and Mary McGaw Professorship in Mathematical Sciences, which supported this research,
T. Bhattacharyya is supported by a J C Bose Fellowship JCB/2021/000041 of SERB, A. Jindal is supported by the Prime Minister's Research Fellowship PM/MHRD-20-15227.03, and P. Kumar is partially supported by a
PIMS postdoctoral fellowship. This research is supported by the DST FIST program-2021 [TPN-700661].

\bibliographystyle{plain}

\end{document}